\documentclass[preprint,12pt]{elsarticle}

\usepackage[utf8]{inputenc}
\usepackage[T1]{fontenc}
\usepackage{amsmath,amssymb,amsthm}
\usepackage{booktabs}
\usepackage{enumitem}
\usepackage[bookmarksopen=true,colorlinks=true,linkcolor=red,citecolor=blue,backref=page]{hyperref}
\usepackage[capitalise,noabbrev]{cleveref}
\usepackage{float}
\usepackage{tikz}
\usetikzlibrary{decorations.pathmorphing,calc}

\usepackage{caption}
\usepackage{amsthm}
\usepackage{tikz-cd}

\usepackage{tkz-graph}
\usetikzlibrary{arrows}

\usepackage{subcaption}

\theoremstyle{plain}
\newtheorem{theorem}{Theorem}[section]
\newtheorem{lemma}[theorem]{Lemma}
\newtheorem{corollary}[theorem]{Corollary}
\newtheorem{proposition}[theorem]{Proposition}

\newcommand{\SD}{\operatorname{SD}}

\begin{document}
\begin{frontmatter}

\title{Generalized Sterboul--Deming Configurations}

\author[addr1,addr2]{Daniel A. Jaume}
\ead{djaume@unsl.edu.ar}
\author[addr1]{Cristian Panelo}
\ead{crpanelo@unsl.edu.ar}
\author[addr1]{Kevin Pereyra}
\ead{kdpereyra@unsl.edu.ar}
\address[addr1]{Departamento de Matem\'{a}tica, Universidad Nacional de San Luis, San Luis, Argentina}
\address[addr2]{Instituto de Matem\'{a}tica Aplicada San Luis, Universidad Nacional de San Luis and CONICET, San Luis, Argentina}

\begin{abstract}
Sterboul and Deming gave classical matching-based characterizations of
non-K\H{o}nig--Egerv\'{a}ry graphs through flower--posy and blossom-pair
configurations. We consider two classical configuration families, denoted
\(T\) and \(S\), and introduce a new walk-based family \(J\), based on
\(J\)-flowers and \(J\)-posies.

Our main result proves that, for every graph \(G\),
\[
\SD_T(G)=\SD_S(G)=\SD_J(G).
\]
Thus the additional flexibility of the \(J\)-framework preserves the set of
vertices detected by the classical configurations. The proof is
vertex-preserving and passes through strict-Hall structure in traces of
\(J\)-posies. As a consequence, every prescribed vertex of a connected
matchable strict-Hall graph lies in a rigid \(T\)-posy for a suitable perfect
matching, linking the theory naturally with matching-covered graphs.
\end{abstract}


\end{frontmatter}

\section{Introduction}\label{sec_introduction}

Let \(\alpha(G)\) be the independence number and \(\mu(G)\) the matching
number of a finite simple graph \(G\). The graph is K\H{o}nig--Egerv\'{a}ry
(KE) if
\[
\alpha(G)+\mu(G)=|V(G)|.
\]
Every bipartite graph is KE by the classical theorem of
Egerv\'{a}ry~\cite{egervary1931combinatorial}. For general graphs, Sterboul
characterized the obstruction to the KE property by excluded configurations
relative to a maximum matching, using Edmonds' flowers and introducing the
posy~\cite{sterboul1979characterization}. Independently, Deming obtained a
matching-based characterization in terms of blossom-pairs generated by
maximum matchings~\cite{deming1979independence}.

We organize the classical path-based constructions into two precise families,
denoted \(T\) and \(S\). We then introduce a new walk-based family, denoted
\(J\). Its basic objects, the \(J\)-flower and the \(J\)-posy, replace selected
alternating paths by alternating walks and allow repeated vertices, repeated
edges, and more flexible intersections between blossoms and joining paths. This
flexibility provides a useful tool for handling alternating structures.

For each of the three configuration families, let \(\SD_T(G)\),
\(\SD_S(G)\), and \(\SD_J(G)\) denote the sets of vertices contained in a
corresponding configuration for some maximum matching of \(G\). Our main
result proves that these three sets coincide. Thus the set of vertices
detected by generalized Sterboul--Deming configurations is independent of the
chosen configuration family.

An important advantage of the configuration-invariance theorem is
methodological. The rigid \(T\)-configurations are particularly convenient
for understanding the underlying obstruction: their geometry is simple,
explicit, and easy to visualize. The more flexible \(J\)-configurations, on
the other hand, are better suited to proofs and structural manipulations,
since alternating walks may repeat vertices and edges and may interact freely
with the underlying blossoms. The theorem therefore separates these two
roles: one may reason and formulate structural questions in terms of the
simpler \(T\)-configurations, while proving the corresponding statements
through the more flexible \(J\)-configurations, without changing the set of
vertices ultimately detected.

The proof is vertex-preserving. Its main ingredients are a structural
analysis of \(J\)-posy traces, an odd-ear absorption argument, and a
reduction of \(J\)-flowers to the posy case. The argument also yields
structural consequences beyond the invariance theorem. In particular, every
prescribed vertex of a connected matchable graph satisfying the strict Hall
condition lies in a \(T\)-posy for a suitable perfect matching. This creates
a natural connection with matching-covered graphs and ear-decomposition
methods.

The paper is organized as follows. \Cref{sec_preliminaries} fixes standard
notation. \Cref{sec_generalized_configurations} introduces the \(T\)-,
\(S\)-, and \(J\)-families of generalized Sterboul--Deming configurations.
\Cref{sec_jposy_traces} studies traces of \(J\)-posies together with
the strict-Hall structure associated with those traces. \Cref{sec_rigid_posies}
proves that every prescribed vertex of
a connected matchable strict-Hall graph lies in a rigid \(T\)-posy for a
suitable perfect matching. \Cref{sec_flower_reduction} reduces \(J\)-flowers
to the rigid framework. \Cref{sec_configuration_invariance} combines these
reductions to prove the configuration-invariance theorem, and
\cref{sec_applications} records further consequences and connections with
matching-covered graphs.

\section{Preliminaries}\label{sec_preliminaries}

Throughout the paper, all graphs are finite, undirected, and simple. For a
graph \(G\), write \(V(G)\) and \(E(G)\) for its vertex and edge sets. For
\(X\subseteq V(G)\), let \(N_G(X)\) denote its open neighborhood and let
\(G-X:=G[V(G)\setminus X]\). When the underlying graph is clear, we omit the
subscript \(G\).

A set \(S\subseteq V(G)\) is independent if no two vertices of \(S\) are
adjacent. A matching is a set of pairwise disjoint edges. A vertex is
\(M\)-\emph{saturated} if it is incident with an edge of \(M\), and
\(M\)-\emph{exposed} otherwise. A matching is perfect if it saturates every
vertex and maximum if it has largest possible cardinality. We denote the
family of maximum matchings of \(G\) by \(\mathcal M(G)\).

Following the usual matching-theoretic convention, we use a matching
interchangeably as an edge set, as the spanning subgraph determined by those
edges, and as an involution. Thus \(M(v)=u\) means that \(uv\in M\), while
\(M(v)=v\) when \(v\) is \(M\)-exposed.

An \(M\)-alternating path or walk alternates between edges in \(M\) and
edges outside \(M\). The length \(|L|\) of a path or walk \(L\) is its number
of edges.

\section{Generalized Sterboul--Deming configurations}
\label{sec_generalized_configurations}

The \(J\)-flower and \(J\)-posy introduced below are new configurations.

Let \(M\) be a matching of \(G\). An edge of \(M\) has status \(m\), while an
edge of \(E(G)\setminus M\) has status \(n\). The type of a non-trivial
\(M\)-alternating path or walk is determined by the status of its first and
last edges in the displayed orientation; thus its type is one of
\[
mm,\qquad mn,\qquad nm,\qquad nn.
\]
  An \(M\)-\emph{blossom} is an odd cycle
\(C\) with exactly \((|C|-1)/2\) edges in \(M\).  Its \emph{base} is the
unique vertex not saturated by the matching edges of \(C\).

An \(M\)-\emph{T-flower} consists of an \(M\)-blossom with base \(b\) and an even alternating path, called the \emph{stem}, from \(b\) to an \(M\)-exposed root \(r\).  The path is allowed to be trivial, meets the blossom only at \(b\), and, when oriented from \(b\) to \(r\), has type \(mn\) when it is non-trivial.

The distinction between the path-based and walk-based objects is illustrated in \cref{fig_allposies}.

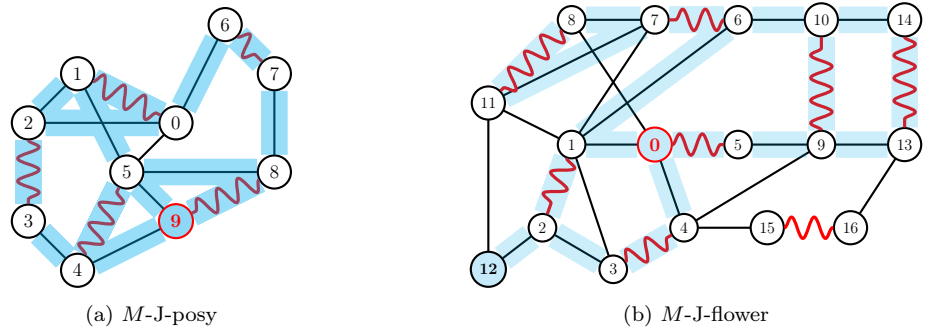
\begin{figure}[H]
	\centering
	\begin{subfigure}{0.45\textwidth}
		\centering
		\def \opacidad{0.2}
		\begin{tikzpicture}[thick,scale=0.65,baseline=(0.base)]
			\node[draw,circle, scale=0.6] (0) at (-2,1) {0};
			\node[draw,circle,scale=0.6] (1) at (-4,2) {1};
			\node[draw,circle,scale=0.6] (2) at (-5,1) {2};
			\node[draw,circle,scale=0.6] (3) at (-5,-1) {3};
			\node[draw,circle,scale=0.6] (4) at (-4,-2) {4};
			\node[draw,circle,scale=0.6] (5) at (-3,0) {5};
			\node[draw,circle,scale=0.6] (6) at (-1,3) {6};
			\node[draw,circle,scale=0.6] (7) at (0,2) {7};
			\node[draw,circle,scale=0.6] (8) at (0,0) {8};
			\node[draw,red,circle,scale=0.6,fill = cyan!40!white] (9) at (-2,-1) {{$\mathbf{9}$}};
			
			\foreach \from/\to in {1/2,1/5,3/4,0/6,7/8,9/5,9/4,0/5,5/8,2/0} {
				\path[draw] (\from) -- (\to);
			}
			\foreach \from/\to in {2/3,4/5,1/0,6/7,8/9} {
				\path [draw, decorate, decoration={snake, segment length=3mm, amplitude=4}, very thick, red] (\from) -- (\to);
			}
			\draw[cyan, opacity=2*\opacidad, line width=10pt, rounded corners] 
			(9) -- (5) -- (4) -- (9) -- (8) -- (5)-- (4)-- (3)-- (2)-- (0)-- (1)-- (2)-- (3)-- (4)-- (5)-- (1)-- (0)-- (6)-- (7)-- (8)-- (9);
		\end{tikzpicture}
		\caption{\(M\)-J-posy}
		\label{fig_Jposy}
	\end{subfigure}
	\hfill
	\begin{subfigure}{0.5\textwidth}
		\centering
		\def \opacidad{0.2}
		\begin{tikzpicture}[thick,scale=0.55,baseline=(0.base)]
			\node[draw,red,circle,scale=0.6,fill = cyan!20!white] (0) at (0,0) {$\mathbf{0}$};
			\node[draw,circle,scale=0.5] (1) at (-2,0) {1};
			\node[draw,circle,scale=0.5] (2) at (-2.7,-2) {2};
			\node[draw,circle,scale=0.5] (3) at (-1,-3) {3};
			\node[draw,circle,scale=0.5] (4) at (0.7,-2) {4};
			\node[draw,circle,scale=0.5] (5) at (2,0) {5};
			\node[draw,circle,scale=0.5] (6) at (2,2+1) {6};
			\node[draw,circle,scale=0.5] (7) at (0,2+1) {7};
			\node[draw,circle,scale=0.5] (8) at (-2,2+1) {8};
			\node[draw,circle,scale=0.5] (9) at (4,0) {9};
			\node[draw,circle,scale=0.5] (10) at (4,2+1) {10};
			\node[draw,circle,scale=0.5] (11) at (-4,1) {11};
			\node[draw,circle,scale=0.5,fill = cyan!20!white] (12) at (-4,-3) {$\mathbf{12}$};
			\node[draw,circle,scale=0.5] (13) at (6,0) {13};
			\node[draw,circle,scale=0.5] (14) at (6,2+1) {14};
			\node[draw,circle,scale=0.5] (P) at (2.7,-2) {15};
			\node[draw,circle,scale=0.5] (Q) at (4.7,-2) {16};
			
			\foreach \from/\to in {1/0,0/4,2/3,7/8,5/9,6/10,11/1,2/12,9/13,10/14,4/P,Q/13,11/12,0/8,4/9,1/3,1/7,11/7,6/1} {
				\path[draw, thick] (\from) -- (\to);
			}
			\foreach \from/\to in {0/5,{1/2},3/4,6/7,9/10,8/11,14/13,P/Q} {
				\path [draw, decorate, decoration={snake, segment length=3mm,amplitude=4}, very thick, red] (\from) -- (\to);
			}
			
			\draw[cyan, opacity=\opacidad, line width=10pt, rounded corners] (0) -- (1) -- (2) -- (3) -- (4) -- (0)-- (5)-- (9)-- (10)-- (14)-- (13)-- (9)-- (10)-- (6)-- (7)-- (8)-- (11)-- (7)-- (6)-- (1)--(2)--(12);
		\end{tikzpicture}
		\caption{$M$-J-flower}\label{fig_Jflower}
	\end{subfigure}
	\caption{Examples of posy and flower configurations}
	\label{fig_allposies}
\end{figure}

An \(M\)-\emph{J-flower} is defined in the same way, except that the stem is an alternating walk, possibly trivial, and has type \(mn\) when non-trivial.  It may repeat vertices and edges and may meet the blossom away from its base. In \cref{fig_Jflower}, a maximum matching is shown in red. The marked $M$-J-flower is formed by the blossom \( 0,1,2,3,4,0 \), with base \( 0 \), and the alternating walk from vertex \( 0 \) to vertex \( 12 \) given by the sequence
\[
0,5,9,10,14,13,9,10,14,13,9,10,6,7,8,11,7,6,1,2,12.
\] 

An \(M\)-\emph{S-posy} consists of two distinct \(M\)-blossoms, not necessarily vertex-disjoint, with distinct bases \(b_1,b_2\), and a non-trivial simple alternating path of type \(mm\) from \(b_1\) to \(b_2\). The joining path may meet either blossom away from its endpoints. An \(M\)-\emph{T-posy} is an S-posy in which the joining path meets the two blossoms only at its endpoints and the union has one of the two standard forms: the blossoms are vertex-disjoint, or their intersection is one non-trivial alternating path inherited by both cycles.  In the second form, the common path has type \(mm\), neither base lies on it, and the four remaining cycle arcs from its ends to the two bases have type \(nn\); the joining path is the opposite \(mm\)-link.  Thus a T-posy is an even subdivision of a barbell or of \(K_4\), respectively.  The matching inherited from \(M\) is perfect on its vertex set.

An \(M\)-\emph{J-posy} consists of two, possibly identical, \(M\)-blossoms whose bases are joined by a non-trivial alternating walk of type \(mm\). The walk may repeat vertices and edges and may meet the blossoms arbitrarily.  The marked $M$-J-posy in \cref{fig_Jposy} is formed by the blossom \( 9,5,4,9 \), whose base is vertex \( 9 \), and the alternating walk:
\[
9,8,5,4,3,2,0,1,2,3,4,5,1,0,6,7,8,9.
\]

An \emph{even subdivision} of a graph \(G\) is a graph obtained from \(G\) by subdividing some of its edges, each an even number of times; equivalently, each edge of \(G\) is replaced by a path of odd length, with paths of length one allowed.  The base barbell is formed by two disjoint triangles joined by one edge; its even subdivisions are two disjoint odd cycles joined by an odd path.

For a graph \(G\), define
\begin{equation}\label{eq_SD_vertex_sets}
\begin{aligned}
\SD_T(G)=\{v:\;&v\text{ lies in an }M\text{-T-flower or }M\text{-T-posy}\\
             &\text{for some }M\in\mathcal M(G)\},\\
\SD_S(G)=\{v:\;&v\text{ lies in an }M\text{-T-flower or }M\text{-S-posy}\\
             &\text{for some }M\in\mathcal M(G)\},\\
\SD_J(G)=\{v:\;&v\text{ lies in an }M\text{-J-flower or }M\text{-J-posy}\\
             &\text{for some }M\in\mathcal M(G)\}.
\end{aligned}
\end{equation}
The notation in \cref{eq_SD_vertex_sets} records the vertices detected by the
three configuration families. Directly from the definitions,
\begin{equation}\label{eq_SD_inclusions}
\SD_T(G)\subseteq\SD_S(G)\subseteq\SD_J(G).
\end{equation}
The main theorem will show that both inclusions in
\cref{eq_SD_inclusions} are equalities.

The main advantage of this generalization is the flexibility of J-posy configurations. For instance, with respect to the red matching in \cref{fig_Jposy_1}, there is an \(M\)-J-posy consisting of the \(M\)-blossoms \(0,1,6,7,2,0\) with base \(0\), and \(5,6,1,2,7,5\) with base \(5\), together with the alternating path $0,3,4,5$. For this matching, however, there is no \(M\)-T-posy.

A second example is shown in \cref{fig_Jposy_2}. There is an \(M\)-J-posy consisting of the \(M\)-blossoms \(0,1,2,0\) with base \(0\), and \(11,12,13,11\) with base \(11\), together with the alternating walk
\[
0,3,4,5,6,7,5,4,8,11,10,9,8,11.
\]
This single \(M\)-J-posy covers every vertex of the graph.

The preceding two examples are displayed together in \cref{fig_allposies_1}.

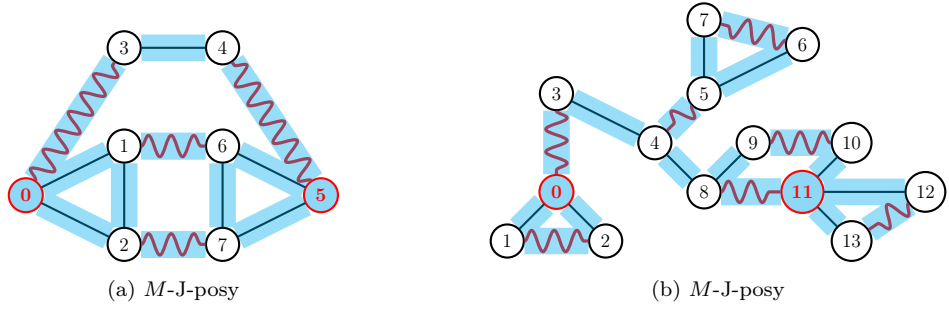
\begin{figure}[H]
	\centering
	\begin{subfigure}{0.45\textwidth}
		\centering
		\def \opacidad{0.2}
		\begin{tikzpicture}[thick,scale=0.65,baseline=(0.base)]
			\node[draw,red,circle,scale=0.6,fill = cyan!40!white] (1) at (-2,1) {{$\mathbf{0}$}};
			\node[draw,circle,scale=0.6] (2) at (0,2) {1};
			\node[draw,circle,scale=0.6] (3) at (0,0) {2};
			\node[draw,circle,scale=0.6] (4) at (0,4) {3};
			\node[draw,circle,scale=0.6] (5) at (2,4) {4};
			\node[draw,red,circle,scale=0.6,fill = cyan!40!white] (6) at (4,1) {{$\mathbf{5}$}};
			\node[draw,circle,scale=0.6] (7) at (2,2) {6};
			\node[draw,circle,scale=0.6] (8) at (2,0) {7};

			\foreach \from/\to in {1/2,1/3,3/2,4/5,7/6,6/8,7/8} {
				\path[draw] (\from) -- (\to);
			}
			\foreach \from/\to in {2/7,3/8,1/4,6/5} {
				\path [draw, decorate, decoration={snake, segment length=3mm, amplitude=4}, very thick, red] (\from) -- (\to);
			}
			\draw[cyan, opacity=2*\opacidad, line width=10pt, rounded corners] 
			(1) -- (2) -- (7) -- (8) -- (3) -- (1)-- (4)-- (5)-- (6)-- (7)-- (2)-- (3)-- (8)-- (6);
		\end{tikzpicture}
		\caption{\(M\)-J-posy}
		\label{fig_Jposy_1}
	\end{subfigure}
	\hfill
	\begin{subfigure}{0.5\textwidth}
		\centering
		\def \opacidad{0.2}
		\begin{tikzpicture}[thick,scale=0.65,baseline=(0.base)]
			\node[draw,red,circle,scale=0.6,fill = cyan!40!white] (0) at (-3,1) {{$\mathbf{0}$}};
			\node[draw,circle,scale=0.6] (1) at (-4,0) {1};
			\node[draw,circle,scale=0.6] (2) at (-2,0) {2};
			\node[draw,circle,scale=0.6] (3) at (-3,3) {3};
			\node[draw,circle,scale=0.6] (4) at (-1,2) {4};
			\node[draw,circle,scale=0.6] (5) at (0,3) {5};
			\node[draw,circle,scale=0.6] (6) at (2,4) {6};
			\node[draw,circle,scale=0.6] (7) at (0,4.5) {7};
			\node[draw,circle,scale=0.6] (8) at (0,1) {8};
			\node[draw,circle,scale=0.6] (9) at (1,2) {9};
			\node[draw,circle,scale=0.6] (10) at (3,2) {10};
			\node[draw,red,circle,scale=0.6,fill = cyan!40!white] (11) at (2,1) {{$\mathbf{11}$}};
			\node[draw,circle,scale=0.6] (12) at (4.5,1) {12};
			\node[draw,circle,scale=0.6] (13) at (3,0) {13};

			\foreach \from/\to in {1/0,0/2,3/4,4/8,5/6,5/7,9/8,11/10,11/12,11/13} {
				\path[draw] (\from) -- (\to);
			}
			\foreach \from/\to in {1/2,3/0,5/4,6/7,8/11,9/10,12/13} {
				\path [draw, decorate, decoration={snake, segment length=3mm, amplitude=4}, very thick, red] (\from) -- (\to);
			}
			\draw[cyan, opacity=2*\opacidad, line width=10pt, rounded corners] 
			(0)-- (1) -- (2) -- (0) -- (3) -- (4) -- (5)-- (6)-- (7)-- (5)-- (4)-- (8)-- (11)-- (10)-- (9)-- (8)-- (11)-- (12)-- (13)-- (11);
		\end{tikzpicture}
		\caption{$M$-J-posy}\label{fig_Jposy_2}
	\end{subfigure}
	\caption{Examples of posy configurations}
	\label{fig_allposies_1}
\end{figure}

\section{Traces of \(J\)-posies and strict Hall structure}
\label{sec_jposy_traces}

The \emph{trace} of a configuration is the subgraph formed by all vertices
and edges used by the configuration.

\begin{lemma}\label{lem_trace_perfect}
Let \(H\) be the trace of an \(M\)-\(J\)-posy and set
\[
M_H:=M\cap E(H).
\]
Then \(M_H\) is a perfect matching of \(H\). Consequently, no edge of \(M\)
has exactly one endpoint in \(V(H)\).
\end{lemma}

\begin{proof}
There are four possible sources of overlap to consider: repeated vertices
of the joining walk, intersections between the joining walk and a blossom,
intersections between the two blossoms, and coincident bases or blossoms.
In each case, any edges of \(M\) incident with the same vertex must
coincide, since \(M\) is a matching.

Let \(v\in V(H)\). If \(v\) occurs in one of the blossoms away from its
base, then the blossom contains an edge of \(M\) incident with \(v\).
If \(v\) occurs on the joining walk, then, since the walk is
\(M\)-alternating of type \(mm\), every occurrence of \(v\), including
an endpoint, is incident with an edge of \(M\) used by the walk.
Thus every vertex of \(H\) is incident with an edge of
\(M\cap E(H)\). By the preceding observation, overlaps among the
constituent objects cannot produce distinct edges of \(M\) at the same
vertex. Hence \(M_H\) is a perfect matching of \(H\).

If an edge of \(M\) had exactly one endpoint in \(V(H)\), its endpoint in
\(H\) would be incident with two distinct edges of \(M\), one in \(M_H\)
and one leaving \(H\), a contradiction.
\end{proof}

Let \(G\) have a perfect matching \(M\). The \emph{alternating digraph}
\(D_M(G)\) has vertex set \(V(G)\), with
\[
x\longrightarrow y
\quad\Longleftrightarrow\quad
xM(y)\in E(G)\setminus M.
\]
Thus an arc \(x\to y\) records the two-edge \(M\)-alternating step
\(x,M(y),y\), whose first edge lies outside \(M\) and whose second edge
lies in \(M\).

For \(X\subseteq V(G)\), write
\[
M(X):=\{M(x):x\in X\}.
\]
Observe that
\[
x\longrightarrow y
\quad\Longleftrightarrow\quad
M(y)\longrightarrow M(x).
\]
Thus the involution \(M\) reverses directed walks in \(D_M(G)\).

A graph \(G\) satisfies the \emph{strict Hall condition} if
\begin{equation}\label{eq_strict_hall}
 |N_G(S)|>|S|
\end{equation}
for every nonempty independent set \(S\subseteq V(G)\). We shall simply
say that \(G\) satisfies \emph{strict Hall}.

\begin{lemma}[Sink Component Lemma]\label{lem_sink_component}
Let \(G\) be a connected graph with a perfect matching \(M\), and let
\(K\) be a sink strong component of \(D_M(G)\). Then either
\[
K=V(G),
\]
or \(K\) is independent and
\[
N_G(K)=M(K).
\]
\end{lemma}

\begin{proof}
We first note that \(M(K)\) is itself a strong component. Indeed, if
\(u,v\in K\), then directed walks in both directions between \(u\) and
\(v\) are reversed by \(M\), so \(M(u)\) and \(M(v)\) are mutually
reachable. Hence \(M(K)\) is strongly connected. If a vertex \(z\) were
mutually reachable with \(M(K)\), then \(M(z)\) would be mutually
reachable with \(K\). By maximality of \(K\), we would have \(M(z)\in K\),
and therefore \(z\in M(K)\).

Suppose first that \(K\cap M(K)\neq\varnothing\). Since \(K\) and
\(M(K)\) are strong components, they must coincide. Thus \(M(K)=K\).
No edge of \(M\) can leave \(K\). If a nonmatching edge \(xy\) had
\(x\in K\) and \(y\notin K\), then
\[
x\longrightarrow M(y).
\]
Since \(K\) is a sink component, \(M(y)\in K\). But \(M(K)=K\), so
\(y\in K\), a contradiction. Hence no edge of \(G\) leaves \(K\).
Since \(G\) is connected, \(K=V(G)\).

Suppose now that \(K\cap M(K)=\varnothing\). We claim that \(K\) is
independent. A matching edge with both endpoints in \(K\) would place one
of them in \(M(K)\), a contradiction. If \(xy\in E(G)\setminus M\) with
\(x,y\in K\), then
\[
x\longrightarrow M(y).
\]
Since \(K\) is a sink component, \(M(y)\in K\), again contradicting
\(K\cap M(K)=\varnothing\). Thus \(K\) is independent.

It remains to determine its neighborhood. Let \(y\in N_G(K)\), and choose
\(x\in K\) adjacent to \(y\). If \(xy\in M\), then \(y=M(x)\in M(K)\).
If \(xy\notin M\), then \(x\to M(y)\), and the sink property gives
\(M(y)\in K\), whence \(y\in M(K)\). Therefore
\[
N_G(K)\subseteq M(K).
\]
The reverse inclusion is immediate because \(M(x)\) is adjacent to \(x\)
for every \(x\in K\). Hence \(N_G(K)=M(K)\).
\end{proof}

\begin{lemma}[Strict Hall--Strong Connectivity Lemma]
\label{lem_strict_hall_strong}
Let \(G\) be a connected graph with a perfect matching \(M\). Then
\begin{equation}\label{eq_strict_hall_strong}
G\text{ satisfies strict Hall}
\quad\Longleftrightarrow\quad
D_M(G)\text{ is strongly connected}.
\end{equation}
\end{lemma}

\begin{proof}
Suppose first that \(D_M(G)\) is strongly connected. Let
\(S\subseteq V(G)\) be a nonempty independent set. Since \(M\) is
perfect,
\[
M(S)\subseteq N_G(S)
\qquad\text{and}\qquad
|M(S)|=|S|.
\]
Thus \(|N_G(S)|\ge |S|\). Suppose equality holds. Then
\[
N_G(S)=M(S).
\]
If \(x\to y\) with \(x\in S\), then
\(xM(y)\in E(G)\setminus M\), so
\[
M(y)\in N_G(S)=M(S).
\]
Since \(M\) is an involution, \(y\in S\). Thus no arc of \(D_M(G)\)
leaves \(S\). Since \(M\) is perfect, the independent set \(S\) is a
proper subset of \(V(G)\), contradicting strong connectivity. Therefore
\cref{eq_strict_hall} holds.

Conversely, suppose that \(G\) satisfies strict Hall but \(D_M(G)\) is
not strongly connected. Choose a sink strong component \(K\). It is
proper, so by \cref{lem_sink_component}, \(K\) is independent and
\[
N_G(K)=M(K).
\]
Consequently,
\[
|N_G(K)|=|M(K)|=|K|,
\]
contradicting strict Hall.
\end{proof}

The next two lemmas describe how the basic pieces of a \(J\)-posy appear
inside the alternating digraph.

\begin{lemma}[Blossom Projection Lemma]\label{lem_blossom_projection}
Let \(G\) have a perfect matching \(M\), and let \(C\) be an
\(M\)-blossom with base \(b\). Put \(p=M(b)\). Then \(D_M(G)\) contains
two directed \(b\)-\(p\) paths whose union contains every vertex of \(C\).
\end{lemma}

\begin{proof}
Write
\[
C=b,v_1,v_2,\ldots,v_{2t},b
\]
so that
\[
v_1v_2,v_3v_4,\ldots,v_{2t-1}v_{2t}\in M.
\]
The remaining edges of \(C\) lie outside \(M\). Since \(p=M(b)\), the
definition of \(D_M(G)\) gives the directed paths
\[
b\to v_2\to v_4\to\cdots\to v_{2t}\to p
\]
and
\[
b\to v_{2t-1}\to v_{2t-3}\to\cdots\to v_1\to p.
\]
Their union contains every vertex of \(C\).
\end{proof}

\begin{lemma}[Alternating-Walk Projection Lemma]
\label{lem_walk_projection}
Let \(G\) have a perfect matching \(M\), and let
\[
W=u_0,u_1,\ldots,u_{2\ell+1}
\]
be an \(M\)-alternating walk of type \(mm\), in its displayed
orientation. Then \(D_M(G)\) contains the directed walks
\[
u_1\to u_3\to\cdots\to u_{2\ell+1}
\]
and
\[
u_{2\ell}\to u_{2\ell-2}\to\cdots\to u_0.
\]
Together they contain every vertex occurring in \(W\).
\end{lemma}

\begin{proof}
Since \(W\) has type \(mm\),
\[
u_0u_1,u_2u_3,\ldots,u_{2\ell}u_{2\ell+1}\in M,
\]
while the intervening edges lie outside \(M\). Hence, for each applicable
index,
\[
u_{2i+1}\to u_{2i+3}
\qquad\text{and}\qquad
u_{2i}\to u_{2i-2}.
\]
These arcs give the two asserted directed walks. Their vertex occurrences
have respectively odd and even indices, and hence together contain every
vertex occurring in \(W\). The statement remains valid when vertices or
edges of \(W\) are repeated; when \(\ell=0\), the two projected walks are
trivial.
\end{proof}

\begin{theorem}\label{thm_jposy_trace_strict_hall}
Let \(H\) be the trace of an \(M\)-\(J\)-posy. Then \(H\) is connected,
\(M_H\) is a perfect matching of \(H\), and
\(D_{M_H}(H)\) is strongly connected. Consequently, \(H\) satisfies the
strict Hall condition.
\end{theorem}

\begin{proof}
Connectedness follows from the definition of a \(J\)-posy, and
\cref{lem_trace_perfect} gives that \(M_H\) is a perfect matching of
\(H\).

Let the two blossoms have bases \(b_1,b_2\), and put
\[
p_i:=M_H(b_i),\qquad i=1,2.
\]
Since every edge of the two blossoms and of the joining walk belongs to
\(H\), its \(M\)-status is unchanged under restriction to \(M_H\). Hence
the two \(M\)-blossoms are also \(M_H\)-blossoms, the joining walk remains
\(M_H\)-alternating of type \(mm\), and \(M_H(b_i)=M(b_i)\) for
\(i=1,2\).

Since the joining walk has type \(mm\), its first edge is \(b_1p_1\) and
its last edge is \(p_2b_2\). By \cref{lem_walk_projection}, the alternating
digraph \(D_{M_H}(H)\) contains directed walks
\[
p_1\leadsto b_2
\qquad\text{and}\qquad
p_2\leadsto b_1.
\]
By \cref{lem_blossom_projection}, it also contains directed paths
\[
b_1\leadsto p_1
\qquad\text{and}\qquad
b_2\leadsto p_2.
\]
Thus \(b_1,p_1,b_2,p_2\) lie in a common strong component \(K\) of
\(D_{M_H}(H)\).

Let \(v\) be a vertex of either blossom. By
\cref{lem_blossom_projection}, \(v\) lies on a directed
\(b_i\)-\(p_i\) path for the corresponding blossom. Since both endpoints
belong to \(K\), so does \(v\).

Now let \(v\) be a vertex occurring on the joining walk. By
\cref{lem_walk_projection}, \(v\) lies on one of the directed walks
\[
p_1\leadsto b_2
\qquad\text{or}\qquad
p_2\leadsto b_1.
\]
Again both endpoints belong to \(K\), and hence \(v\in K\).

Every vertex of the trace belongs to one of the two blossoms or to the
joining walk. Therefore \(K=V(H)\), and \(D_{M_H}(H)\) is strongly
connected. The strict Hall condition now follows from
\cref{lem_strict_hall_strong}.

The argument does not require the two blossoms to be disjoint or distinct,
nor the joining walk to be simple. Coincident bases, intersections between
the blossoms and the joining walk, and repeated vertices or edges merely
identify vertices already lying in the same strong component.
\end{proof}

\section{Rigid posies in strict-Hall graphs}\label{sec_rigid_posies}

The finite parity analysis in the next lemma uses the two rigid geometries of a T-posy. The \(K_4\)-type geometry and the three orbits of ear endpoints used there are displayed in \cref{fig_k4_tposy}.

\begin{figure}[H]
\centering
\begin{subfigure}{0.61\textwidth}
\centering
\def\opacidad{0.20}
\begin{tikzpicture}[thick,scale=.93,baseline=(current bounding box.center),
    branch/.style={draw,circle,fill=white,inner sep=1.4pt,font=\scriptsize},
    sub/.style={draw,circle,fill=white,inner sep=.7pt},
    medge/.style={draw,decorate,decoration={snake,segment length=2.4mm,amplitude=2.7},very thick,red},
    nedge/.style={draw,thick}]
  \begin{scope}[rotate=25]
    \coordinate (v1) at (0,2.65);
    \coordinate (v2) at (-2.65,0);
    \coordinate (v3) at (2.65,0);
    \coordinate (v4) at (0,-2.45);

    \path (v1)--(v2) coordinate[pos=.34] (a12) coordinate[pos=.68] (b12);
    \path (v1)--(v3) coordinate[pos=.34] (a13) coordinate[pos=.68] (b13);
    \path (v1)--(v4) coordinate[pos=.34] (a14) coordinate[pos=.68] (b14);
    \path (v2)--(v3) coordinate[pos=.34] (a23) coordinate[pos=.68] (b23);
    \path (v2)--(v4) coordinate[pos=.34] (a24) coordinate[pos=.68] (b24);
    \path (v3)--(v4) coordinate[pos=.34] (a34) coordinate[pos=.68] (b34);

    \draw[cyan,opacity=.23,line width=12pt,rounded corners]
      (v1)--(a12)--(b12)--(v2)--(a23)--(b23)--(v3)--(b13)--(a13)--(v1);
    \draw[green!60!black,opacity=.18,line width=12pt,rounded corners]
      (v1)--(a14)--(b14)--(v4)--(b34)--(a34)--(v3)--(b13)--(a13)--(v1);
    \draw[blue!70,opacity=.22,line width=9pt,rounded corners]
      (v2)--(a24)--(b24)--(v4);

    \foreach \u/\a/\b/\v in {v1/a12/b12/v2,v1/a14/b14/v4,v2/a23/b23/v3,v3/a34/b34/v4}{
      \draw[nedge] (\u)--(\a);
      \draw[medge] (\a)--(\b);
      \draw[nedge] (\b)--(\v);
    }
    \foreach \u/\a/\b/\v in {v1/a13/b13/v3,v2/a24/b24/v4}{
      \draw[medge] (\u)--(\a);
      \draw[nedge] (\a)--(\b);
      \draw[medge] (\b)--(\v);
    }

    \node[branch] at (v1) {$1$};
    \node[branch] at (v2) {$2$};
    \node[branch] at (v3) {$3$};
    \node[branch] at (v4) {$4$};
    \foreach \p in {a12,b12,a13,b13,a14,b14,a23,b23,a24,b24,a34,b34}
      \node[sub] at (\p) {};

    \node[font=\scriptsize,fill=white,inner sep=1pt] at ($(v1)!0.50!(v2)+(-.16,.15)$) {$L_{12}$};
    \node[font=\scriptsize,fill=white,inner sep=1pt] at ($(v1)!0.51!(v3)+(.16,.18)$) {$L_{13}$};
    \node[font=\scriptsize,fill=white,inner sep=1pt] at ($(v1)!0.52!(v4)+(.22,0)$) {$L_{14}$};
    \node[font=\scriptsize,fill=white,inner sep=1pt] at ($(v2)!0.50!(v3)+(0,.18)$) {$L_{23}$};
    \node[font=\scriptsize,fill=white,inner sep=1pt] at ($(v2)!0.50!(v4)+(-.18,-.10)$) {$L_{24}$};
    \node[font=\scriptsize,fill=white,inner sep=1pt] at ($(v3)!0.50!(v4)+(.18,-.10)$) {$L_{34}$};
  \end{scope}

  \node[font=\scriptsize,text=cyan!55!black] at (-2.55,2.80) {$B_1$};
  \node[font=\scriptsize,text=green!45!black] at (2.55,-2.35) {$B_2$};
  \node[font=\scriptsize,text=blue!70!black,align=center] at (-2.50,-2.45) {opposite $mm$\\joining path};
\end{tikzpicture}
\caption{A $K_4$-type T-posy.  The cyan and green blossoms share one $mm$ link, while the blue opposite $mm$ link joins their bases; the other four links are of type $nn$.}

\end{subfigure}\hfill
\begin{subfigure}{0.35\textwidth}
\centering
\begin{tikzpicture}[scale=.80,thick,
    branch/.style={draw,circle,fill=white,inner sep=1pt},
    markx/.style={circle,fill=blue!65,inner sep=1.7pt},
    marky/.style={circle,fill=green!55!black,inner sep=1.7pt}]
  \foreach \yy/\lab in {3.05/{same},0/{adjacent},-3.05/{opposite}}{
    \coordinate (p1) at (0,\yy+1.00);
    \coordinate (p2) at (-1.10,\yy);
    \coordinate (p3) at (1.10,\yy);
    \coordinate (p4) at (0,\yy-1.00);
    \draw (p1)--(p2)--(p4)--(p3)--(p1) (p2)--(p3) (p1)--(p4);
    \foreach \p in {p1,p2,p3,p4}{\node[branch] at (\p) {};}
    \node[font=\scriptsize] at (1.83,\yy) {\lab};
    \ifdim\yy pt>2pt
      \draw[blue!65,line width=2.3pt] (p1)--(p2);
      \node[markx,label={[font=\tiny]145:$x$}] at ($(p1)!.34!(p2)$) {};
      \node[marky,label={[font=\tiny]220:$y$}] at ($(p1)!.69!(p2)$) {};
    \else\ifdim\yy pt>-2pt
      \draw[blue!65,line width=2.3pt] (p1)--(p2);
      \draw[green!55!black,line width=2.3pt] (p1)--(p3);
      \node[markx,label={[font=\tiny]145:$x$}] at ($(p1)!.56!(p2)$) {};
      \node[marky,label={[font=\tiny]35:$y$}] at ($(p1)!.56!(p3)$) {};
    \else
      \draw[blue!65,line width=2.3pt] (p1)--(p2);
      \draw[green!55!black,line width=2.3pt] (p3)--(p4);
      \node[markx,label={[font=\tiny]145:$x$}] at ($(p1)!.56!(p2)$) {};
      \node[marky,label={[font=\tiny]-35:$y$}] at ($(p3)!.56!(p4)$) {};
    \fi\fi
  }
\end{tikzpicture}
\caption{The three link-pair orbits for the endpoints $x,y$ of the odd ear: same, adjacent, and opposite.}

\end{subfigure}
\caption{The $K_4$ geometry used in the T-posy and Odd-Ear arguments.  Matching edges are drawn as red wavy edges, as in \cref{fig_allposies}; panel (b) records the three symmetry classes that exhaust the $K_4$ case of Odd-Ear Absorption.}
\label{fig_k4_tposy}
\end{figure}
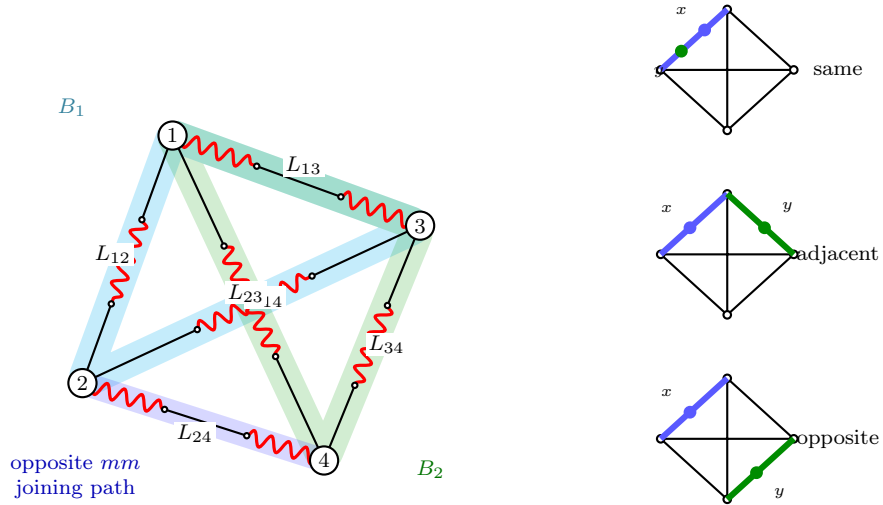

The next lemma contains the only finite structural analysis required in the proof.  It is stated
with its matching-extension clause because that clause is essential when the
T-posy lies inside a larger graph.

\begin{lemma}[Odd-Ear Absorption Lemma]
\label{lem_ear_absorption}
Let \(Q\) be an \(M\)-\(T\)-posy in a graph \(G\), and let
\(R=xRy\) be a simple odd \(M\)-alternating ear relative to \(Q\),
whose internal vertices lie outside \(Q\), and whose first and last
edges have status \(n\). We allow \(x=y\), in which case \(R\) is a
simple closed ear.

Set
\[
X:=V(Q)\cup V(R).
\]
Then there exists a perfect matching \(\widehat M\) of \(G\), agreeing
with \(M\) outside \(G[X]\), and an \(\widehat M\)-\(T\)-posy
\(Q'\subseteq G[X]\) such that
\[
\operatorname{int}R\subseteq V(Q').
\]
\end{lemma}

\begin{proof}
We first isolate the matching surgery used throughout the proof.
Since \(Q\) is an \(M\)-\(T\)-posy, every vertex of \(Q\) is
\(M\)-saturated inside \(Q\). Since \(R\) is an odd alternating ear of
type \(nn\), all internal vertices of \(R\) are paired by \(M\)-edges
of \(R\). Hence every vertex of \(X\) has its \(M\)-mate in \(X\), and
therefore no edge of \(M\) has exactly one endpoint in \(X\).

Consequently, whenever we construct a perfect matching
\(\widehat M_X\) of \(G[X]\), it extends to a perfect matching of \(G\)
by
\[
\widehat M
=
\widehat M_X
\cup
\bigl(M\setminus E(G[X])\bigr).
\]
Thus, in each case below, it is enough to exhibit a \(T\)-posy
\(Q'\subseteq G[X]\) containing \(\operatorname{int}R\), together with
a perfect matching of the residual vertices \(X\setminus V(Q')\).
We give \(Q'\) its canonical perfect matching and match the residual
pieces consecutively along the paths described below.

Suppose first that \(x=y\). Then \(R\) is an \(M\)-blossom with base
\(x\). Since \(Q\) is a \(T\)-posy, a direct inspection of its two
possible geometries yields an \(M\)-blossom \(C\subseteq Q\), with base
\(c\), and a nontrivial simple \(M\)-alternating path \(A\) of type
\(mm\) from \(c\) to \(x\), such that
\[
V(A)\cap V(C)=\{c\}.
\]
Indeed, in the barbell case this follows from the parity of the two
routes from \(x\) to the two bases. In the \(K_4\)-subdivision case,
after relabeling so that two opposite links have type \(mm\), the other
four have type \(nn\), and there is an \(M\)-blossom based at each
branch vertex. The four branch-vertex positions, the two \(mm\)-link
positions, and the four \(nn\)-link positions are all covered by the
same parity routing: for an \(mm\)-link use the odd terminal segment;
for an \(nn\)-link prepend the appropriate adjacent \(mm\)-link to the
even terminal segment. In every case the chosen path meets its blossom
only at the base.

Since \(V(R)\cap V(Q)=\{x\}\), the blossoms \(C\) and \(R\) are
vertex-disjoint and \(A\) meets them only at their bases. Hence
\[
Q':=C\cup A\cup R
\]
is an \(M\)-\(T\)-posy in the barbell case, and no change of matching
is required. Hence assume from now on that \(x\neq y\).

Suppose first that \(Q\) is in the barbell case. Write
\[
Q=C_1\cup P\cup C_2,
\]
where \(C_1,C_2\) are the two blossoms, with bases \(b_1,b_2\), and
\(P=b_1Pb_2\) is the joining path. Up to interchanging the two blossoms
and the two ends of \(R\), there are four possible positions of
\(x\) and \(y\).

If \(x,y\in V(P)\), let \(S=xPy\). If \(\|S\|\) is odd, replace \(S\)
by \(R\). The resulting \(b_1\)-\(b_2\) path remains odd and yields a
new barbell \(T\)-posy. The residual vertices are precisely
\(\operatorname{int}S\), which have even cardinality. If \(\|S\|\) is
even, then
\[
D:=R\cup S
\]
is an odd cycle. Of the two tails \(b_1Px\) and \(yPb_2\), exactly one
has odd length. Joining \(D\) to the blossom at the end of that odd
tail yields a barbell \(T\)-posy. The remaining tail has even length;
together with the unused original blossom it leaves only path pieces
of even order, and these admit consecutive perfect matchings.

If \(x\in V(C_1)\) and \(y\in V(P)\), write
\[
a:=\|b_1Py\|,
\qquad
d:=\|yPb_2\|.
\]
The numbers \(a\) and \(d\) have opposite parity. If \(d\) is even,
then
\[
R\cup yPb_2
\]
is an odd path from \(x\) to \(b_2\), and
\[
C_1\cup R\cup yPb_2\cup C_2
\]
is a barbell with bases \(x\) and \(b_2\). The residual is
\(\operatorname{int}(b_1Py)\), which has even order. If \(d\) is odd,
then \(a\) is even. Let \(E\) be the even \(b_1\)-\(x\) arc of \(C_1\)
and put
\[
D:=R\cup yPb_1\cup E.
\]
Then \(D\) is an odd cycle and
\[
D\cup yPb_2\cup C_2
\]
is a barbell. The complementary \(b_1\)-\(x\) arc of \(C_1\) has odd
length, so its internal vertices have even cardinality. The case
\(x=b_1\) is included by taking \(E\) to be the trivial path.

If \(x,y\in V(C_1)\), let \(E\) and \(F\) be respectively the even and
odd \(x\)-\(y\) arcs of \(C_1\). Then
\[
D:=R\cup E
\]
is an odd cycle. If \(b_1\in V(E)\), then
\[
D\cup P\cup C_2
\]
is a barbell, and the residual is \(\operatorname{int}F\), which has
even order. If \(b_1\in V(F)\), write
\[
F=xFb_1Fy.
\]
Exactly one of the two segments \(xFb_1\) and \(b_1Fy\) has even
length. Choosing the corresponding endpoint \(x\) or \(y\) as the base
of \(D\), and joining it to \(C_2\) through that even segment followed
by \(P\), produces an odd joining path. The unused complementary
segment has odd length and hence an even number of internal vertices.

Finally, if \(x\in V(C_1)\) and \(y\in V(C_2)\), then
\[
Q':=C_1\cup R\cup C_2
\]
is itself a barbell with bases \(x\) and \(y\). Give this barbell its
canonical perfect matching. The only residual vertices are
\(\operatorname{int}P\), and they have even cardinality because
\(\|P\|\) is odd.

This exhausts the barbell case.

Now suppose that \(Q\) is an even subdivision of \(K_4\). Label its
branch vertices \(1,2,3,4\) and its links \(L_{ij}\). All six links
have odd length. Up to an automorphism of \(K_4\), the ends \(x,y\)
lie on the same link, on adjacent links, or on opposite links. Branch
vertices are included in these cases; trivial terminal subpaths of
length \(0\) are allowed.

Assume first that \(x,y\in L_{12}\), in this order, and let
\[
S:=xL_{12}y.
\]
If \(\|S\|\) is odd, replace \(S\) by \(R\). The resulting six links
are again all odd, so the resulting graph is an even subdivision of
\(K_4\). The residual is \(\operatorname{int}S\), which has even
cardinality. If \(\|S\|\) is even, then
\[
D:=R\cup S
\]
is an odd cycle. The two tails \(1L_{12}x\) and \(yL_{12}2\) have
opposite parity. The odd tail joins \(D\) to the appropriate odd cycle
through the remaining four branch vertices, producing a barbell
\(T\)-posy. The complementary tail has even length, and together with
the unused link interiors yields only path components of even order.

Next suppose that \(x\in L_{12}\) and \(y\in L_{13}\). Put
\[
\alpha:=\|1L_{12}x\|\pmod2,
\qquad
\beta:=\|1L_{13}y\|\pmod2.
\]
There are four parity pairs \((\alpha,\beta)\). In the case
\((0,0)\), the odd cycle
\[
R\cup xL_{12}1L_{13}y
\]
is joined through \(L_{14}\) to
\[
L_{23}\cup L_{34}\cup L_{42},
\]
and the residual consists of the interiors of \(xL_{12}2\) and
\(yL_{13}3\). In the other three parity cases the witness is an even
subdivision of \(K_4\): the two relevant odd cycles share one odd link,
the four remaining cycle arms are odd, and the joining path is odd.
For \((1,1)\), \((1,0)\), and \((0,1)\), respectively, the residual is
the interior of \(L_{23}\), \(1L_{12}x\), and \(1L_{13}y\). Thus the
four parity cases are all covered, and every residual component has even
order.

Finally suppose that \(x\in L_{12}\) and \(y\in L_{34}\). Let
\(i,i'\in\{1,2\}\) be chosen so that
\[
\|xL_{12}i\|\equiv1,
\qquad
\|xL_{12}i'\|\equiv0
\pmod2,
\]
and let \(j,j'\in\{3,4\}\) be chosen so that
\[
\|yL_{34}j\|\equiv0,
\qquad
\|yL_{34}j'\|\equiv1
\pmod2.
\]
Set
\[
C=
R\cup xL_{12}i\cup L_{ij}\cup jL_{34}y,
\]
\[
D=
L_{12}\cup L_{1j'}\cup L_{2j'},
\]
and
\[
A=yL_{34}j'.
\]
The cycles \(C\) and \(D\) are odd, they share the odd path
\(xL_{12}i\), and \(A\) is odd. More explicitly, the six effective
links have branch vertices \(x,i,y,j'\) and are
\[
xL_{12}i,
\qquad
yL_{34}j',
\qquad
R,
\qquad
L_{ij}\cup jL_{34}y,
\qquad
L_{ij'},
\qquad
xL_{12}i'\cup L_{i'j'}.
\]
Each has odd length. Hence their union is an even subdivision of
\(K_4\), and therefore a \(T\)-posy containing all internal vertices
of \(R\). The only unused original link is \(L_{i'j}\); its internal
vertices have even cardinality and admit a consecutive perfect
matching.

The same constructions remain valid when an end of \(R\) is a base or
a branch vertex; the corresponding terminal subpath is simply trivial.
Thus all possible positions of the ends of \(R\) are covered.

In every case we have obtained a \(T\)-posy \(Q'\subseteq G[X]\)
containing \(\operatorname{int}R\) and a perfect matching of the
residual vertices. Combining the canonical perfect matching of \(Q'\)
with the residual matching gives a perfect matching
\(\widehat M_X\) of \(G[X]\), and the initial matching surgery extends
it to a perfect matching \(\widehat M\) of \(G\) agreeing with \(M\)
outside \(G[X]\).
\end{proof}

\begin{lemma}[Directed-Path Lifting Lemma]
\label{lem_directed_path_lifting}
Let \(G\) have a perfect matching \(M\), let \(Q\) be an
\(M\)-saturated subgraph of \(G\), and let
\(xy\in E(G)\setminus M\) with \(x\in V(Q)\) and \(y\notin V(Q)\).
Put \(z_0=M(y)\), and let
\[
z_0,z_1,\ldots,z_k=q
\]
be a shortest directed path in \(D_M(G)\) from \(z_0\) to \(V(Q)\),
chosen so that \(q\) is its first vertex in \(Q\). Then \(k\geq1\), and
\[
W=x,y,z_0,M(z_1),z_1,\ldots,z_{k-1},M(q)
\]
is an odd \(M\)-alternating walk of type \(nn\). Moreover,
\[
y\in\operatorname{int}W
\qquad\text{and}\qquad
V(W)\cap V(Q)=\{x,M(q)\},
\]
where the two vertices in the last set may coincide.
\end{lemma}

\begin{proof}
Since \(z_0\notin V(Q)\), the directed path has positive length, so
\(k\geq1\). For each \(0\leq i<k\), the arc \(z_i\to z_{i+1}\) means
\[
z_iM(z_{i+1})\in E(G)\setminus M.
\]
Together with \(xy\notin M\), \(yz_0=yM(y)\in M\), and the matching
edges \(M(z_i)z_i\), this gives the status sequence
\[
n,m,n,m,\ldots,m,n
\]
along \(W\). Hence \(W\) has type \(nn\), and
\(\|W\|=2k+1\) is odd. Since \(k\geq1\), the vertex \(y\) is internal.

By the choice of the directed path,
\[
z_0,z_1,\ldots,z_{k-1}\notin V(Q),
\qquad q\in V(Q).
\]
Because \(Q\) is \(M\)-saturated, \(M(q)\in V(Q)\). If
\(M(z_i)\in V(Q)\) for some \(1\leq i<k\), then the unique
\(M\)-mate of \(M(z_i)\) also lies in \(Q\), forcing
\(z_i\in V(Q)\), a contradiction. Thus no intermediate vertex of the
lift lies in \(Q\), and the asserted intersection follows.
\end{proof}

\begin{theorem}[Rigid Posy Theorem]\label{thm_marked_strict_hall}
Let \(H\) be a connected graph with a perfect matching and suppose that
\(H\) satisfies strict Hall. Then every vertex \(v\in V(H)\) lies in an
\(M_v\)-T-posy for some perfect matching \(M_v\) of \(H\).
\end{theorem}

\begin{proof}
Write \(|V(H)|=2h\), and let \(S\) be a maximum independent set. Since
\(S\) is maximal,
\[
N_H(S)=V(H)\setminus S.
\]
Strict Hall gives \(2h-|S|>|S|\), and hence
\[
\alpha(H)<h=\mu(H).
\]
Thus \(H\) is not K\H{o}nig--Egerv\'{a}ry. By the characterization of
factorizable K\H{o}nig--Egerv\'{a}ry graphs by forbidden nice subgraphs
\cite{lovasz1983ear,lovasz2009matching}, \(H\) contains as a nice subgraph
an even subdivision \(Q\) of either the barbell graph or \(K_4\). Give
\(Q\) its canonical perfect matching and combine it with a perfect matching
of \(H-V(Q)\). This yields a perfect matching \(M\) of \(H\) for which
\(Q\) is an \(M\)-T-posy. In particular, at least one vertex of \(H\) lies
in a T-posy for some perfect matching.

Let \(Z\) be the set of vertices of \(H\) that lie in an
\(M'\)-T-posy for some perfect matching \(M'\) of \(H\).
Suppose that \(Z\neq V(H)\). Since \(H\) is connected, there is an edge
\(xy\) with \(x\in Z\) and \(y\notin Z\). Choose a perfect matching
\(M\) of \(H\) and an \(M\)-T-posy \(Q\) containing \(x\). Since every
vertex of \(Q\) is \(M\)-saturated inside \(Q\), we have \(xy\notin M\);
otherwise \(y=M(x)\in V(Q)\subseteq Z\). Put \(z_0=M(y)\). Again
\(z_0\notin V(Q)\), for otherwise \(y=M(z_0)\in V(Q)\).

By \cref{lem_strict_hall_strong}, \(D_M(H)\) is strongly connected. Choose
a shortest directed path from \(z_0\) to \(V(Q)\), and apply
\cref{lem_directed_path_lifting}. We obtain an odd \(M\)-alternating walk
\(W\) of type \(nn\) such that
\[
y\in\operatorname{int}W,
\qquad
V(W)\cap V(Q)=\{x,M(q)\}.
\]

If \(W\) has no repeated vertex, except possibly that its two endpoints
coincide, then \(W\) is a simple odd ear relative to \(Q\). By
\cref{lem_ear_absorption}, there is a T-posy containing \(y\) for some
perfect matching of \(H\), contradicting \(y\notin Z\).

Suppose therefore that \(W\) repeats a vertex. In the displayed lift, the
even-position vertices
\[
x,z_0,\ldots,z_{k-1}
\]
are distinct, and the odd-position vertices
\[
y=M(z_0),M(z_1),\ldots,M(q)
\]
are distinct. Hence every repetition identifies an even-position vertex
with an odd-position vertex. Consider the first second occurrence of a
vertex along \(W\). It cannot occur in an even position. Indeed, if
\(z_i=M(z_j)\) with \(j<i\), then \(M(z_i)=z_j\) already gives a second
occurrence one position earlier; the case \(i=j\) is impossible because
\(M\) is perfect in a simple graph. Therefore the first repeated segment
starts in an even position and ends in an odd position. It has odd length,
its first and last edges have status \(n\), and it has no internal
repetition. Thus it is an \(M\)-blossom \(B\), with base \(b\).

Let \(P\) be the prefix of \(W\) from \(x\) to the first occurrence of
\(b\). Then \(P\) is either trivial or a simple \(M\)-alternating path of
type \(nm\), and
\[
V(P)\cap V(B)=\{b\}.
\]
Moreover, \(y\in V(P\cup B)\). We claim that
\[
V(P\cup B)\cap V(Q)=\{x\}.
\]
The only possible vertex of \(P\cup B\) in \(Q\), apart from \(x\), is
the final endpoint \(M(q)\) of \(W\). If \(B\) reaches \(M(q)\), then
this final occurrence repeats an even-position vertex lying in \(Q\). The
only such vertex is \(x\), so \(b=M(q)=x\). This proves the claim.

It remains to connect the new blossom \(B\) to a blossom already contained
in \(Q\). We use only the local geometry of the T-posy \(Q\). In the
barbell case, parity along the joining path and the two odd cycles yields an
\(M\)-blossom \(C\subseteq Q\), with base \(c\), and a nontrivial simple
\(M\)-alternating path \(A\) of type \(mm\) from \(c\) to \(x\), meeting
\(C\) only at \(c\). In the \(K_4\)-subdivision case, after labeling the
two opposite \(mm\)-links, the same conclusion follows from the routing
through the two \(mm\)-links and four \(nn\)-links used in the closed-ear
case of \cref{lem_ear_absorption}. Thus in either geometry
\[
V(A)\cap V(C)=\{c\}.
\]

If \(b=x\), then \(x\notin V(C)\), since \(A\) is nontrivial and meets
\(C\) only at \(c\). If \(b\neq x\), then \(B\) lies outside \(Q\).
Hence \(C\cap B=\varnothing\). Since \(P\) meets \(Q\) only at \(x\),
the concatenation \(A\cup P\) is a nontrivial simple
\(M\)-alternating path of type \(mm\) from \(c\) to \(b\), and it meets
\(C\cup B\) only at its endpoints. Consequently
\[
C\cup A\cup P\cup B
\]
is an \(M\)-T-posy containing \(y\), again contradicting \(y\notin Z\).
Therefore \(Z=V(H)\), which proves the theorem.
\end{proof}

\begin{theorem}\label{thm_marked_jposy}
Let \(G\) be a graph, let \(M\in\mathcal M(G)\), and let \(H\) be the trace
of an \(M\)-J-posy.  For every \(v\in V(H)\), there are a maximum matching
\(M_v\in\mathcal M(G)\) and an \(M_v\)-T-posy containing \(v\).
\end{theorem}

\begin{proof}
By \cref{thm_jposy_trace_strict_hall}, \(H\) is connected, matchable, and strict Hall.
Apply \cref{thm_marked_strict_hall} to obtain a perfect matching \(M_v'\) of
\(H\) and an \(M_v'\)-T-posy containing \(v\).  By
\cref{lem_trace_perfect}, no edge of \(M\) joins \(V(H)\) to its complement.
Moreover, every \(M\)-edge incident with \(V(H)\) belongs to \(E(H)\), so
\(M\setminus E(H)\) is a matching on \(G-V(H)\) and matches precisely the
vertices there that were saturated by \(M\).  Since
\(|M_v'|=|M\cap E(H)|=|V(H)|/2\),
\[
M_v=M_v'\cup\bigl(M\setminus E(H)\bigr)
\]
is a matching of \(G\) with \(|M_v|=|M|\).  It is therefore maximum and
extends the required T-posy.
\end{proof}

\begin{corollary}\label{cor_posy_invariance}
A vertex lies in an $M$-J-posy for some maximum matching if and only if it lies in a T-posy for some maximum matching.  Hence T-, S-, and J-posies mark the same vertices.
\end{corollary}

\begin{proof}
Only the reverse inclusion needs proof, and it is exactly
\cref{thm_marked_jposy}.
\end{proof}

\section{Reduction of \(J\)-flowers}\label{sec_flower_reduction}

The flower case is reduced to the posy case by a small triangle attachment.
We separate the matching bookkeeping and the removal step from the main
reduction.

\begin{lemma}[Triangle-Attachment Matching Lemma]
\label{lem_triangle_attachment_matching}
Let \(G\) be a graph, let \(M\in\mathcal M(G)\), and let \(r\) be
\(M\)-exposed. Form \(G^+\) by adding three new vertices \(x,y,z\), the
triangle \(xyzx\), and the edge \(rx\). Then
\begin{equation}\label{eq_triangle_matching}
M^+:=M\cup\{rx,yz\}
\end{equation}
is a maximum matching of \(G^+\). Moreover, every maximum matching of
\(G^+\) contains both \(rx\) and \(yz\), and its restriction to \(G\) is a
maximum matching of \(G\).
\end{lemma}

\begin{proof}
Deleting the three new vertices from any matching of \(G^+\) removes at most
two edges. Hence every matching of \(G^+\) has size at most
\(|M|+2=|M^+|\), proving that the matching in
\cref{eq_triangle_matching} is maximum.

Among the four edges incident with the three new vertices, the only two
disjoint edges are \(rx\) and \(yz\). Therefore a maximum matching that
omitted either of them would use at most
one edge incident with the new vertices. After restriction to \(G\), such a
matching would have size at most \(|M|+1\), whereas every maximum matching of
\(G^+\) has size \(|M|+2\). Thus every maximum matching contains \(rx\) and
\(yz\), and deleting these two edges leaves a maximum matching of \(G\).
\end{proof}

\begin{lemma}[Triangle-Removal Lemma]\label{lem_triangle_removal}
Use the notation of \cref{lem_triangle_attachment_matching}, and let
\(N^+\in\mathcal M(G^+)\). Suppose that \(Q\) is an
\(N^+\)-T-posy containing an old vertex \(v\in V(G)\). If \(Q\) meets
\(\{x,y,z\}\), then the triangle \(xyzx\) is one blossom of a barbell-type
T-posy, with base \(x\), and its joining path leaves through \(xr\). Removing
the triangle and \(xr\) leaves an \(N\)-T-flower of \(G\) containing \(v\),
where
\[
N=N^+\setminus\{rx,yz\}.
\]
\end{lemma}

\begin{proof}
Every vertex of a T-posy is matched inside it. Thus the presence of \(y\) or
\(z\) forces the matching edge \(yz\), and their degree two then forces both
\(xy\) and \(xz\). If only \(x\) were present, its matching edge \(xr\) and
the degree condition in a subdivision would force \(xy\) or \(xz\), with the
same conclusion. Hence \(Q\) contains the whole triangle.

The T-posy \(Q\) cannot be of \(K_4\)-type: the cycle \(xyzx\) would contain
at most one branch vertex because \(y,z\) have degree two, whereas every
cycle in a subdivision of \(K_4\) contains at least three branch vertices.
Therefore the triangle is one blossom of a barbell-type T-posy, with base
\(x\), and its joining path leaves through the unique external matching edge
\(xr\).

By \cref{lem_triangle_attachment_matching},
\(N=N^+\setminus\{rx,yz\}\) is maximum in \(G\). Deleting the triangle and
\(xr\) from \(Q\) leaves the other \(N\)-blossom and a simple alternating
path, possibly trivial, from its base to \(r\), internally disjoint from that
blossom. When non-trivial and oriented from the base to \(r\), the path has
type \(mn\), so the remainder is an \(N\)-T-flower. Since \(v\) is an old
vertex, it remains in the resulting configuration.
\end{proof}

\begin{theorem}[\(J\)-Flower Reduction Theorem]\label{thm_marked_jflower}
Let \(G\) be a graph, let \(M\in\mathcal M(G)\), and let \(F\) be an
\(M\)-J-flower with exposed root \(r\). For every \(v\in V(F)\), there are
a maximum matching \(N\in\mathcal M(G)\) and either an \(N\)-T-flower or an
\(N\)-T-posy containing \(v\).
\end{theorem}

\begin{proof}
Let the blossom of \(F\) have base \(b\), and orient its generalized stem
\(W\) from \(b\) to \(r\). Form \(G^+\) as in
\cref{lem_triangle_attachment_matching}. The triangle is an
\(M^+\)-blossom with base \(x\), where \(M^+\) is given by
\cref{eq_triangle_matching}. Following \(W\) from \(b\) to \(r\) and then
the matching edge \(rx\) gives an \(M^+\)-\(mm\)-walk from \(b\) to \(x\).
Hence the original blossom, this walk, and the new triangle form an
\(M^+\)-J-posy.

By \cref{thm_marked_jposy}, the old vertex \(v\) lies in an
\(N^+\)-T-posy \(Q\) of \(G^+\), for some maximum matching \(N^+\).
By \cref{lem_triangle_attachment_matching},
\(N=N^+\setminus\{rx,yz\}\) is a maximum matching of \(G\).

If \(Q\) avoids the triangle, then \(Q\) is already an \(N\)-T-posy of
\(G\). If \(Q\) meets the triangle, \cref{lem_triangle_removal} gives an
\(N\)-T-flower of \(G\) containing \(v\).
\end{proof}

\section{Configuration invariance}\label{sec_configuration_invariance}

\begin{theorem}[Configuration Invariance Theorem]
\label{thm_configuration_invariance}
For every graph \(G\),
\begin{equation}\label{eq_configuration_invariance}
\SD_T(G)=\SD_S(G)=\SD_J(G).
\end{equation}
We therefore write
\[
\SD(G):=\SD_T(G)=\SD_S(G)=\SD_J(G)
\]
for their common vertex set.
\end{theorem}

\begin{proof}
The definitions give
\(\SD_T(G)\subseteq\SD_S(G)\subseteq\SD_J(G)\).  Let
\(v\in\SD_J(G)\).  If \(v\) lies in an $M$-J-posy, the marked reduction
\cref{thm_marked_jposy} gives a T-posy containing \(v\), relative to a
suitable maximum matching.  If \(v\) lies in a J-flower,
\cref{thm_marked_jflower} gives a T-flower or a T-posy containing \(v\), again
relative to a suitable maximum matching.  Thus \(v\in\SD_T(G)\), proving the
reverse inclusion.
\end{proof}

The equality in \cref{eq_configuration_invariance} has a useful
methodological interpretation. The additional freedom of the
\(J\)-configurations is structural rather than set-theoretic: it facilitates
the manipulation of alternating structures without changing the set of
vertices detected. Thus the theorem allows us to work simultaneously at two
different levels of rigidity. The \(T\)-configurations provide a simple
geometric model in which the relevant obstruction is easy to recognize and
reason about, whereas the \(J\)-configurations provide a substantially more
flexible language in which to carry out reductions and proofs. Attempting the
latter arguments directly within the rigid \(T\)-framework would generally
require considerably more delicate control of intersections, alternating
paths, and reroutings.

\begin{corollary}[Configuration-independent Sterboul characterization]
\label{cor_configuration_independent_sterboul}
For a graph \(G\), the following statements are equivalent:
\begin{enumerate}[label=(\roman*)]
\item \(G\) is not a K\H{o}nig--Egerv\'{a}ry graph;
\item for some maximum matching of \(G\), there is a T-configuration;
\item for some maximum matching of \(G\), there is an S-configuration;
\item for some maximum matching of \(G\), there is a J-configuration.
\end{enumerate}
Here an \(X\)-configuration means an \(X\)-flower or an \(X\)-posy, with the
convention that S-flowers are T-flowers.
\end{corollary}

\begin{proof}
The equivalence of (1) and (2) is the classical Sterboul characterization
in the rigid formulation used here; see the original paper
\cite{sterboul1979characterization}. Deming obtained an independent
characterization of K\H{o}nig--Egerv\'{a}ry graphs
\cite{deming1979independence}.  The implications
(2)\(\Rightarrow\)(3)\(\Rightarrow\)(4) follow directly from the definitions.
Finally, (4) implies \(\SD_J(G)\neq\varnothing\), and
\cref{thm_configuration_invariance} gives
\(\SD_T(G)=\SD_J(G)\neq\varnothing\).  Hence a T-configuration exists for
some maximum matching, and (2) follows.
\end{proof}

\medskip
\noindent\textbf{Remark.}
Sterboul's original theorem contains the stronger matching-wise statement:
if \(G\) is non-KE, then every maximum matching admits a classical flower or
posy.  The corollary above records the corresponding vertex-level invariance.
It does not identify the witnessing maximum matchings across the three
conventions.
\medskip

\section{Further consequences and matching-covered graphs}\label{sec_applications}

The arguments used above have a direct interpretation in the theory of
matching-covered graphs.  We record two short consequences.  These consequences are not needed for
the configuration-invariance theorem, but they place the strict Hall condition
and the alternating digraph in a standard matching-theoretic setting.

Recall that a connected graph with at least two vertices is
\emph{matching-covered} if every edge belongs to a perfect matching.

\begin{proposition}\label{prop_strict_hall_matching_covered}
Let \(H\) be a connected graph with a perfect matching.  Then \(H\) satisfies
strict Hall if and only if \(H\) is non-bipartite and matching-covered.
\end{proposition}

\begin{proof}
Suppose first that \(H\) satisfies strict Hall, and fix a perfect matching
\(M\).  By \cref{lem_strict_hall_strong}, \(D_M(H)\) is strongly connected.
Let \(e=xy\notin M\).  The arc of \(D_M(H)\) arising from \(e\) lies on a
directed cycle.  Lifting this cycle to \(H\) gives an \(M\)-alternating cycle
containing \(e\).  Switching \(M\) along this cycle gives a perfect matching
containing \(e\).  The edges of \(M\) already belong to a perfect matching,
so \(H\) is matching-covered.  If \(H\) were bipartite, with bipartition
\((A,B)\), then \(|A|=|B|\) and \(N_H(A)=B\), contradicting strict Hall.

Conversely, suppose that \(H\) is non-bipartite and matching-covered.  Let
\(S\) be a non-empty independent set.  A perfect matching gives
\(|N_H(S)|\ge |S|\).  If equality held, then every perfect matching would
match \(S\) bijectively to \(N_H(S)\).  Hence no perfect matching could use
an edge with both ends in \(N_H(S)\), or an edge joining \(N_H(S)\) to
\(V(H)\setminus(S\cup N_H(S))\).  Since \(H\) is matching-covered, no such
edge exists.  Also, no edge joins \(S\) to
\(V(H)\setminus(S\cup N_H(S))\), by the definition of the neighborhood.
Connectedness now gives \(V(H)=S\cup N_H(S)\).  Moreover, \(N_H(S)\) is
independent, since an edge inside it would belong to no perfect matching.
Thus \((S,N_H(S))\) is a bipartition of \(H\), a contradiction.  Therefore
\(|N_H(S)|>|S|\).
\end{proof}

Combining \cref{prop_strict_hall_matching_covered} with
\cref{lem_strict_hall_strong} removes the apparent dependence on the chosen
perfect matching.

\begin{corollary}\label{cor_alternating_digraph_matching_covered}
Let \(H\) be connected and let \(M\) be a perfect matching of \(H\).  Then
\[
\begin{aligned}
D_M(H)\text{ is strongly connected}
&\quad\Longleftrightarrow\quad
H\text{ is non-bipartite}\\
&\hspace{42mm}\text{and matching-covered}.
\end{aligned}
\]
Consequently, if \(D_M(H)\) is strongly connected for one perfect matching
\(M\), then \(D_{M'}(H)\) is strongly connected for every perfect matching
\(M'\) of \(H\).
\end{corollary}

In particular, \cref{thm_jposy_trace_strict_hall} shows that the trace of every J-posy
is a non-bipartite matching-covered graph.  Thus the passage from a J-posy to
its trace replaces a flexible alternating-walk configuration by a classical
matching-covered object.

This passage also illustrates one of the practical advantages of configuration
invariance. The rigid \(T\)-configurations are well suited to recognizing and
formulating the obstruction, but they are considerably less convenient for
structural arguments involving traces, alternating reroutings, or conformal
reductions. The J-posy framework absorbs precisely this extra freedom, allowing
the proof to move into the language of strict Hall and matching-covered graphs
while preserving the same vertex set detected by the original
\(T\)-configurations.

We next connect the vertex-prescribed viewpoint with ear decompositions.  The proposition below is a convenient consequence of classical theory of ear decompositions for matching-covered graphs.  We claim no novelty for the existence of the conformal subdivisions used there.  A
matching-covered subgraph \(K\) of a matching-covered graph \(H\) is
\emph{conformal} if \(H-V(K)\) has a perfect matching.  A cycle \(C\) of
\(H\) is \emph{conformal} if \(H-V(C)\) has a perfect matching.  Lov\'asz's
ear-decomposition theorem states that a matching-covered subgraph is conformal
exactly when it occurs as a member of some ear decomposition of the ambient
graph; see~\cite{lovasz1983ear}.  We also use Little's classical theorem that
any two edges of a matching-covered graph lie in a common conformal cycle
\cite{little1974factorcovered}.  As usual in the matching-covered literature,
\(\Theta\) denotes the two-vertex multigraph with three parallel edges.  In an
ear decomposition starting from a conformal cycle, the first proper extension
consists either of one odd ear, producing an even subdivision of \(\Theta\),
or of a double ear, producing an even subdivision of \(K_4\); this is the
elementary first-extension case of the ear-decomposition theorem for
matching-covered graphs~\cite{lovasz1983ear}.

\begin{proposition}\label{prop_vertex_prescribed_ear_reduction}
Let \(H\) be a matching-covered graph that is neither \(K_2\) nor an even
cycle.  For every vertex \(v\in V(H)\), there is a conformal matching-covered
subgraph \(K_v\) containing \(v\) such that \(K_v\) is an even subdivision
of \(\Theta\) or \(K_4\).  Moreover, \(H\) has an ear decomposition that
contains \(K_v\); equivalently, reversing this ear decomposition reduces
\(H\) to \(K_v\) by successively deleting an ear or a double ear while
remaining within the class of matching-covered graphs.
\end{proposition}

\begin{proof}
Since \(H\neq K_2\) is matching-covered, \(v\) is incident with at least
two distinct edges.  Choose two such edges.  By Little's theorem
\cite{little1974factorcovered}, they lie in a common conformal cycle \(C\).
Since \(H\) is not an even cycle, an ear decomposition of \(H\) that contains
\(C\) has a first proper extension \(K_v\) of \(C\).  If this extension adds
one ear, then \(K_v\) is an even subdivision of \(\Theta\).  If it adds a
double ear, then \(K_v\) is an even subdivision of \(K_4\).  In either case
\(C\subseteq K_v\), and hence \(v\in V(K_v)\).  Every member of the ear
decomposition is conformal in \(H\). Reversing the portion of the decomposition
from \(H\) back to \(K_v\) therefore yields the asserted reduction.
\end{proof}

The first application identifies the strict Hall condition arising naturally
from J-posy traces with the classical class of non-bipartite matching-covered
graphs.  The second retains the vertex-prescribed character of our arguments:
a prescribed vertex can be preserved while the ambient matching-covered graph
is reduced to a conformal even subdivision of \(\Theta\) or \(K_4\), one of the
elementary structures arising at the beginning of an ear decomposition.
Stronger local and conformal-minor results are available in the
matching-covered literature; see, for example, \cite{kothari2026cubic}.

Together, these consequences illustrate the methodological role of
configuration invariance.  The rigid \(T\)-configurations provide the simpler
objects in which to recognize and formulate the relevant obstructions,
whereas the \(J\)-configurations provide the flexibility needed for structural
transformations and reductions.  Configuration invariance allows one to move
between these two levels without losing the vertex information carried by the
original configuration.

\section*{Acknowledgments}
This work was partially supported by Universidad Nacional de San Luis
(Argentina), PROICO 03-0723, MATH AmSud grant 22-MATH-02, Agencia I+D+i
(Argentina) grants PICT-2020-Serie A-00549 and PICT-2021-CAT-II-00105, and
CONICET (Argentina) grant PIP 11220220100068CO\@.

\section*{Declaration of generative AI and AI-assisted technologies in the writing process}
During the preparation of this work, the authors used ChatGPT (OpenAI) to improve the grammar and organization of several passages. After using this service, the authors reviewed and edited the content as needed and take full responsibility for the content of the publication.

\section*{Data availability}
Data sharing is not applicable to this article because no datasets were
generated or analyzed.

\section*{Declarations}
\textbf{Conflict of interest.} The authors declare that they have no conflict
of interest.

\bibliographystyle{elsarticle-num}
\bibliography{TAGmaster}

\end{document}